\documentclass[12pt]{article}
\usepackage{amssymb, amsmath, amsthm}
\usepackage[left=1.2in, right=1.2in, top=1.1in, bottom=1.2in]{geometry}
\usepackage{hyperref}
\usepackage{tikz}
\usepackage{graphics}

\newtheorem{theorem}{Theorem}[section]
\newtheorem{lemma}[theorem]{Lemma}
\newtheorem{corollary}[theorem]{Corollary}

\newtheorem{claim}[theorem]{Claim}
\newtheorem{conjecture}[theorem]{Conjecture}

\newtheorem{problem}[theorem]{Problem}
\newtheorem{question}[theorem]{Question}

\def\beq{\begin{equation}}\def\eeq{\end{equation}}
\def\beqn{\begin{eqnarray}}\def\eeqn{\end{eqnarray}}

\newcommand{\avdeg}{\underline{\deg}}

\title{Large monochromatic components in multicolored bipartite graphs}

\author{Louis DeBiasio\thanks{Department of Mathematics, Miami University, Oxford, Ohio. \texttt{debiasld@miamioh.edu}, \texttt{kruegera@miamioh.edu}} \thanks{Research supported in part by Simons Foundation Collaboration Grant \#283194}
\and Robert A. Krueger\footnotemark[1]
\and G\'{a}bor N. S\'ark\"ozy\thanks{Alfr\'ed R\'enyi Institute of Mathematics, Hungarian Academy of Sciences, Budapest, P.O. Box 127, Budapest, Hungary, H-1364. \texttt{sarkozy.gabor@renyi.mta.hu}}
\thanks{Computer Science Department, Worcester Polytechnic Institute, Worcester, MA.} \thanks{Research supported in part by
NKFIH Grants No. K116769, K117879.}
}

\date{}
\begin{document}
 \maketitle

\begin{abstract}
It is well-known that in every $r$-coloring of the edges of the complete bipartite
graph $K_{m,n}$ there is a monochromatic connected component with at least ${m+n\over r}$ vertices.
In this paper we study an extension of this problem by replacing complete bipartite graphs by bipartite graphs of large minimum degree. We conjecture that in every $r$-coloring of the edges of
an $(X,Y)$-bipartite graph with $|X|=m$, $|Y|=n$, $\delta(X,Y) > \left( 1 - \frac{1}{r+1}\right) n$ and $\delta(Y,X) > \left( 1 - \frac{1}{r+1}\right) m$,
there exists a monochromatic component on at least $\frac{m+n}{r}$ vertices (as in the complete bipartite graph).
If true, the minimum degree condition is sharp (in that both inequalities cannot be made weak when $m$ and $n$ are divisible by $r+1$).

We prove the conjecture for $r=2$ and we prove a weaker bound for all $r\geq 3$.  As a corollary, we obtain a result about the existence of monochromatic components with at least $\frac{n}{r-1}$ vertices in $r$-colored graphs with large minimum degree.
\end{abstract}

\section{Introduction, results}

We let $V(G)$ and $E(G)$ denote the vertex-set and the edge-set of the
graph $G$, $e(G)=|E(G)|$.
$N_G(v)$ is the set of neighbours of $v$.
Hence $|N_G(v)|=\deg_G(v)$, the degree of $v$.
If $G$ is an $(X,Y)$-bipartite graph, then the minimum
degree from $X$ to $Y$ (from $Y$ to $X$) is denoted by $\delta_G(X,Y)$
($\delta_G(Y,X)$). Furthermore, the average degree from $X$ to $Y$, $\avdeg_G(X,Y)$,
is the average of the degrees in $X$, i.e. $\sum_{v\in X}\deg_G(v)/|X|$ and similarly for
$\avdeg_G(Y,X)$.
We may omit the subscript $G$ if it is clear from the context.  Given $X'\subseteq X$ and $Y'\subseteq Y$, we write $[X', Y']$ to denote the subgraph induced by $X'\cup Y'$ and $e(X', Y')$ to denote the number of edges in $[X', Y']$.  

There are many results about monochromatic connected components of edge colored graphs and hypergraphs (see the surveys \cite{GYSUR}, \cite{GYSUR2}, and \cite{KL}).  For example, the following is well-known.

\begin{theorem}[\cite{GY1}]\label{gy} In every  $r$-edge coloring of a complete graph on $n$ vertices
there is a monochromatic connected component of order at least ${n\over r-1}$.
\end{theorem}

In this paper connected components of a graph are just called {\em components} and in edge-colored graphs {\em monochromatic components} are the components of the graph defined by the edges of the same color. 

Recently there has been significant interest in extending Theorem \ref{gy} to non-complete host graphs 
(e.g. 
\cite{BD}, 
\cite{DP}, 
\cite{GYS}, 
\cite{GYS2}, 
\cite{S}).
In particular, in \cite{GYS2} the authors studied the extension of Theorem \ref{gy} to $r$-edge colored graphs of large minimum degree.

In the case where the host graph is a complete bipartite graph (see \cite[Section 3.1]{GYSUR}) the following result provides an analogue of Theorem \ref{gy}.

\begin{theorem}[\cite{GY1}]\label{gy1} In every $r$-edge coloring of the edges of $K_{m,n}$ there is a monochromatic component with at least ${m+n\over r}$ vertices.
\end{theorem}

Mubayi \cite{MU} and Liu, Morris, and Prince \cite{LMP} independently obtained a stronger result with a clever application of the Cauchy-Schwarz inequality: one can require that the monochromatic component in Lemma \ref{gy1} is a {\em double star} (a tree obtained by joining the centers of two disjoint stars by an edge).

Here we address the natural combination of the above two problems and we study the largest monochromatic component in an $r$-edge colored bipartite graph of large minimum degree.

\begin{question}\label{que} What minimum degree condition will guarantee that an $(X,Y)$-bipartite graph $G$ with $|X|=m$, $|Y|=n$,  has the property that in any $r$-edge coloring of $G$ there is a monochromatic component with at least ${m+n\over r}$ vertices (as in the complete bipartite graph)?
\end{question}

First we provide an answer for $r=2$.

\begin{theorem}\label{r=2}
Let $G$ be an $(X,Y)$-bipartite graph with $|X|=m$, $|Y|=n$. If $\delta(X,Y) > \frac{2}{3} n$ and $\delta(Y,X) > \frac{2}{3} m$,
then in every 2-edge coloring of $G$, there exists a monochromatic component on at least $\frac{m+n}{2}$ vertices.
\end{theorem}

We make the following conjecture for general $r$.

\begin{conjecture}\label{r}
Let $r\geq 3$ and let $G$ be an $(X,Y)$-bipartite graph with $|X|=m$, $|Y|=n$. If $\delta(X,Y) > \left( 1 - \frac{1}{r+1}\right) n$ and $\delta(Y,X) > \left( 1 - \frac{1}{r+1}\right) m$,
then in every $r$-edge coloring of $G$, there exists a monochromatic component on at least $\frac{m+n}{r}$ vertices.
\end{conjecture}

The bounds in Conjecture \ref{r} cannot be improved when $m$ and $n$ are divisible by $(r+1)$.  Consider a 1-factorization of the complete bipartite graph
$K_{(r+1),(r+1)}$ with partite sets $X'$ and $Y'$ where the edges coming from the $i$-th matching, $1\leq i \leq r$, are colored with color $i$ and the edges coming from the
$(r+1)$-st matching are removed.  Replace each vertex in $X'$ with a set of $t_1$ points (resulting in $X$), each vertex in $Y'$ with a set of $t_2$ points (resulting in $Y$), and each edge of color $i$ with a complete bipartite graph $K_{t_1,t_2}$ of color $i$.
The bipartite graph obtained has $m+n=(r+1)t_1 + (r+1) t_2$ vertices, it has minimum degrees
$\delta(X,Y) = \left( 1 - \frac{1}{r+1}\right) n$ and $\delta(Y,X) = \left( 1 - \frac{1}{r+1}\right) m$,
yet the largest monochromatic component has size only ${m+n\over r+1}$.  Thus if Conjecture \ref{r} is true, there is a jump in the size of the largest monochromatic component just below the given minimum degree threshold.  Also note that it may be possible to slightly refine Conjecture \ref{r} by only requiring $\delta(X,Y) \geq \left( 1 - \frac{1}{r+1}\right) n$ and $\delta(Y,X) \geq \left( 1 - \frac{1}{r+1}\right) m$, provided equality does not hold in both.  

At the moment, we are only able to prove the following weaker version of Conjecture \ref{r}.

\begin{theorem}\label{tetel}
Let $r$ be an integer with $r\geq 2$ and let $m$ and $n$ be integers.  If $G$ is an $X,Y$-bipartite graph with $m=|X|\leq |Y|=n$, $\delta(X,Y) > \left( 1 - \frac{(m/n)^3}{128r^5}\right) n$, and $\delta(Y,X) > \left( 1 - \frac{(m/n)^3}{128r^5}\right) m$,
then in every $r$-edge coloring of $G$, there exists a monochromatic component on at least $\frac{m+n}{r}$ vertices.
\end{theorem}

It is interesting to note that while there is a monochromatic component of the same size as in the complete bipartite graph, we can no longer guarantee that this component is a double star.  Indeed, assume that $m$ and $n$ are divisible by $r$ and consider a 1-factorization of the complete bipartite graph $K_{r,r}$ with partite sets $X'$ and $Y'$, where the edges coming from the $i$-th matching in the 1-factorization, $1\leq i \leq r$, are colored with color $i$.  Let $2\leq t_1\leq t_2$ be positive integers and replace each vertex in $X'$ with a set of $t_1$ points (resulting in $X$), each vertex in $Y'$ with a set of $t_2$ points
(resulting in $Y$), and each edge of color $i$ with a complete bipartite graph $K_{t_1,t_2}$ of color $i$ from which we remove a matching of size $t_1$.  The bipartite graph thus obtained has $m+n=r t_1 + r t_2$ vertices, it has minimum degrees
$\delta(X,Y) = n - r$ and $\delta(Y,X) \geq m - r$, and the largest monochromatic component still has size ${m+n\over r}$ as claimed; however, the largest monochromatic double star only has size at most $t_1-1+t_2={m+n\over r}-1$.


Another natural way to answer Question \ref{que} is to consider an ``additive'' minimum degree condition; that is, a lower bound on $\delta(X,Y)+\delta(Y,X)$.  We prove the following result for two colors; however, we don't believe this degree condition to be best possible.

\begin{theorem}\label{-n/8}
Let $G$ be an $X,Y$-bipartite graph on $n$ vertices with $|Y|\geq |X|> n/4$.  If $\delta(X,Y)\geq |Y|-n/8$ and $\delta(Y,X)\geq |X|-n/8$, then in every 2-coloring of the edges of $G$ there exists a monochromatic component $H$ such that $|H\cap X|\geq |X|/2$ and $|H\cap Y|\geq |Y|/2$; so in particular, $|H|\geq n/2$.
\end{theorem}

\subsection{Graphs with large minimum degree}

In \cite{GYS2}, the authors conjecture that for all $r\geq 3$, if $G$ is a graph on $n$ vertices with $\delta(G)\geq (1-\frac{r-1}{r^2})n$, then in every $r$-coloring of the edges of $G$, there is a monochromatic component on at least $\frac{n}{r-1}$ vertices.  Our results for bipartite graphs have some consequences for the graph case.

We first obtain the following corollary of Theorem \ref{tetel} which improves the bound of $(1-\frac{1}{1000(r-1)^9})n$ given in \cite{GYS2}.

\begin{corollary}\label{cor}
Let $r$ be an integer with $r\geq 3$ and let $G$ be a graph on $n$ vertices.  If $\delta(G)\geq (1-\frac{1}{3072(r-1)^5})n$, then in every $r$-coloring of  the edges of $G$, there exists a monochromatic component on at least $\frac{n}{r-1}$ vertices.
\end{corollary}

\begin{proof}
Suppose $\delta(G)\geq (1-\frac{1}{3072(r-1)^5})n$ and consider an $r$-edge coloring of $G$.  If there exists a monochromatic component of order at least $\frac{n}{r-1}$, then we are done, so suppose not.  Then there is a bipartition $\{X,Y\}$ of $G$ such that $|Y|\geq |X|\geq n/3$ and $G[X,Y]$ is colored with at most $r-1$ colors.  We have $|X|/|Y|\geq 1/2$ and thus 
$$(|X|^3/|Y|^3)|Y|\geq (|X|^3/|Y|^3)|X|\geq n/24.$$ 
So we have
$$\delta(X,Y)\geq |Y|-\frac{n}{3072(r-1)^5}=|Y|-\frac{n/24}{128(r-1)^5}\geq \left(1-\frac{(|X|/|Y|)^3}{128(r-1)^5}\right)|Y|,$$ 
and 
$$\delta(Y,X)\geq |X|-\frac{n}{3072(r-1)^5}=|X|-\frac{n/24}{128(r-1)^5}\geq \left(1-\frac{(|X|/|Y|)^3}{128(r-1)^5}\right)|X|$$ 
so thus may apply Theorem \ref{tetel} to get a monochromatic component in $G[X,Y]$ of size at least $\frac{|X|+|Y|}{r-1}=\frac{n}{r-1}$.
\end{proof}

We also obtain the following corollary of Theorem \ref{-n/8} which improves the bound of $9n/10$ given in \cite{GYS2} (with a different method of proof).

\begin{corollary}\label{7/8}
Let $G$ be a graph with $\delta(G)\geq 7n/8$.  In every 3-coloring of the edges of $G$, there exists a monochromatic component on at least $n/2$ vertices.  
\end{corollary}

\begin{proof}
If there is no monochromatic connected component of order at least $3n/4$ in say green, then there exists a bipartition $\{X,Y\}$ of $V(G)$ with $|X|, |Y|> n/4$ such that $G[X,Y]$ is colored with only red and blue.  Apply Theorem \ref{-n/8} to the 2-colored bipartite graph $G[X,Y]$ to get the desired monochromatic component.  
\end{proof}

We prove Theorem \ref{r=2} in Section \ref{two_colors}, we prove Theorem \ref{tetel} in Section \ref{r_colors}, and we prove Theorem \ref{-n/8} in Section \ref{sec:add}.  

\section{Two colors}\label{two_colors}

\begin{proof}[Proof of Theorem \ref{r=2}]
Let $G$ be an $(X,Y)$-bipartite graph with $|X|=m$, $|Y|=n$, $\delta(X,Y) > \frac{2}{3} n$, $\delta(Y,X) > \frac{2}{3} m$ and 
consider an arbitrary red/blue coloring of the edges of $G$. 

Let $H_1$ be the largest monochromatic, say blue, component.
Let $X_1=H_1\cap X$, $Y_1=H_1\cap Y$, $X_2=X\setminus H_1$, $Y_2=Y\setminus H_1$.  Let $x_i=|X_i|$ and $y_i=|Y_i|$.   

If $y_1\leq n/3$, then every pair of vertices in $X_1$ has a common red neighbor in $Y_2$ and every vertex in $X_1$ has more than $n/3$ red neighbors in $Y_2$, giving us a red component with more than $x_1+n/3\geq x_1+y_1$ vertices; so we have $y_1>n/3$.  Likewise, we have $x_1>m/3$.

\begin{claim}\label{doneif}
If there exists a monochromatic component which intersects $X$ in at least $2m/3$ vertices or $Y$ in at least $2n/3$ vertices, then there exists a monochromatic component of order at least $\frac{m+n}{2}$.  In particular, if there exists a vertex which is only incident with edges of a single color, then there exists a monochromatic component of order at least $\frac{m+n}{2}$.
\end{claim}

\begin{proof}
Suppose there exists some, say blue, component $H_1'$ such that $|H_1'\cap X|\geq 2m/3$ or $|H_1'\cap Y|\geq 2n/3$.  Let $X_1'=H_1'\cap X$, $Y_1'=H_1'\cap Y$, $X_2'=X\setminus H_1'$, $Y_2'=Y\setminus H_1'$.  Let $x_i'=|X_i'|$ and $y_i'=|Y_i'|$.  So if $x_1'\geq 2m/3$, then we have the desired monochromatic component unless $y_1'<n/2$.  But now every vertex in $X_1'$ has a red neighbor in $Y_2'$ and every pair of vertices in $Y_2'$ has a common red neighbor in $X_1'$, so we have a red component of size greater than $2m/3+n/2\geq \frac{n+m}{2}$.  Likewise if $y_1'\geq 2n/3$. 

Since $\delta(X,Y)>2n/3$ and $\delta(Y,X)>2m/3$, if there exists a vertex which is only incident with edges of a single color, this implies there is a monochromatic component which intersects $X$ in at least $2m/3$ vertices or $Y$ in at least $2n/3$ vertices.
\end{proof}

By Claim \ref{doneif}, we may suppose for the remainder of the proof that every vertex is incident with at least one edge of each color.  Also note that for the rest of the proof, if we have at most two monochromatic components covering all of $V(G)$, then we are done since at least one of them has at least $\frac{m+n}{2}$ vertices.  

Since $x_1>m/3$ and $y_2>n/3$, every vertex in $Y_2$ has a red neighbor in $X_1$ and every vertex in $X_1$ has a red neighbor in $Y_2$, so there are non-trivial red components $R_1, \dots, R_k$ in $G$ which cover all of $X_1\cup Y_2$.  Likewise, since $y_1>n/3$ and $x_2>m/3$, there are non-trivial red components $R_1', \dots, R_\ell'$ in $G$ which cover all of $X_2\cup Y_1$.  Note that it is possible to have $R_i=R_j'$ for some $i\in [k]$, $j\in [\ell]$ (in fact, if there is a red edge in $[X_2, Y_2]$ or $[X_1, Y_1]$, this will necessarily be the case).

Suppose first that $\ell=1$.  Of course we must have $k\geq 2$ otherwise there would be at most two red components covering all of $G$ and we are done.  Let $Y_2'=\{v\in Y_2: N_R(v)\cap X_2=\emptyset\}$.  If $|Y_2'|\leq n/3$, then since $R_1'$ covers $X_2$, there is a red component covering $|Y|-|Y_2'|\geq 2n/3$ vertices of $Y$ and we are done by Claim \ref{doneif}.  So suppose $|Y_2'|>n/3$.  If there exists $i\in [k]$ such that $Y_2'\subseteq R_i$, then since for every $j\neq i$, every vertex in $R_j$ has a red neighbor in $R_1'$, there are at most two red components in $G$ covering all of $V(G)$ and we are done; so suppose not.  Since every pair of vertices from $Y_2'$ in different red components $R_i$ and $R_j$ have more than $m/3$ common neighbors, all of which must be in $X_2$ (they have no blue edges to $X_1$ and they are in distinct red components), all of these common neighbors are blue and thus there is a blue component $H_2$ in $[X_2, Y_2]$ covering $Y_2'$.  Furthermore, since $|Y_2'|>n/3$, every vertex in $X_2$ has a (necessarily) blue neighbor in $Y_2'$ and thus $H_2$ covers $X_2$.  Finally, since every vertex in $Y_2$ has at least one blue neighbor (by Claim \ref{doneif}) and it must be in $X_2$, this implies that $H_2$ covers $X_2\cup Y_2$ and thus we have two blue components covering all of $V(G)$ and we are done.  The proof for $k=1$ is analogous, so we have $k\geq 2$ and $\ell\geq 2$.

Suppose next that $\ell=2$. If there is a blue component $H_2$ covering $X_2$, then since every vertex in $Y_2$ has a blue neighbor (by Claim \ref{doneif}), but none in $X_1$, every vertex in $Y_2$ is also covered by $H_2$.  Thus we have two blue components covering all of $V(G)$ and we are done.  So suppose there is no blue component covering $X_2$, which allows us to choose vertices $u\in R_1'\cap X_2$ and $u'\in R_2'\cap X_2$ such that $u$ and $u'$ have no common blue neighbors.  Since $u$ and $u'$ are in different red components and are not in $H_1$, we have $N(u)\cap N(u')\subseteq Y_2$ and $N_R(u)\cap N_R(u')=\emptyset$.  This implies that every common neighbor of $u$ and $u'$ is in $Y_2$ and is either a red neighbor of $u$ or a red neighbor of $u'$; i.e. 
\[
|(N_R(u)\cup N_R(u'))\cap Y_2|\geq |N(u)\cap N(u')|>n/3.
\]
Thus every vertex in $X_1$ has a (necessarily) red neighbor in $(N_R(u)\cup N_R(u'))\cap Y_2$ which implies that there are at most two red components in $G$ covering $X\cup Y_1$.  Finally, since every vertex in $Y_2$ has a red neighbor in $X_1$, this implies that there are at most two red components in $G$ covering all of $V(G)$ and we are done.  The proof for $k= 2$ is analogous. 

Finally, suppose $k\geq 3$ and $\ell\geq 3$.  For all $i\in [k]$, choose $u_i\in R_i\cap X_1$ and $v_i\in R_i\cap Y_2$ and for all $j\in [\ell]$, choose $u_j'\in R_j'\cap X_2$ and $v_j'\in R_j'\cap Y_1$.  We have
\[
x_1+y_2\geq \sum_{i=1}^k(\deg_{X_1}(v_i)+\deg_{Y_2}(u_i))>k(2m/3-x_2+2n/3-y_1)=k(x_1-m/3+y_2-n/3) 
\]
which implies
\begin{equation}\label{x1y2}
x_1+y_2<\frac{k}{3(k-1)}(m+n),
\end{equation}
and 
\[
x_2+y_1\geq \sum_{j=1}^\ell(\deg_{Y_1}(u_j')+\deg_{X_2}(v_j'))>k(2m/3-x_1+2n/3-y_2)=k(x_2-m/3+y_1-n/3) 
\]
which implies
\begin{equation}\label{x2y1}
x_2+y_1<\frac{\ell}{3(\ell-1)}(m+n).
\end{equation}
Since $k\geq 3$ and $\ell\geq 3$, \eqref{x1y2} and \eqref{x2y1} imply 
\[m+n=x_1+y_2+x_2+y_1<(\frac{k}{3(k-1)}+\frac{\ell}{3(\ell-1)})(m+n)\leq m+n,\] a contradiction.
\end{proof}

\section{$r$ colors}\label{r_colors}

\subsection{Stability}

As noted in the introduction, Mubayi, and independently Liu, Morris, and Prince proved a density version of Theorem \ref{gy1}.

\begin{lemma}[\cite{LMP},\cite{MU}]\label{andraslemma2}
Let $0\leq \eta\leq 1$ and let $G$ be an $(X,Y)$-bipartite graph with $|X|=m$, $|Y|=n$.
If $e(G)\geq \eta mn$, then $G$ has a
double star (component) of order at least $\eta(m+n)$.
\end{lemma}

We begin by proving a stability version of this lemma, i.e. either we have a slightly larger double star than guaranteed by Lemma \ref{andraslemma2} or we have strong structural properties in $G$, namely apart from a small number of exceptional vertices all vertices have degrees close to the average degree.

\begin{lemma}\label{stabilitylemma}
Let $r$ be an integer with $r\geq 2$ and $\delta>0$.  Let $G$ be an $(X,Y)$-bipartite graph with $|X|=m$, $|Y|=n$, and $m \leq n$, and set $\alpha=\frac{m+n}{r^2n}\delta$ and $\beta=\frac{m+n}{r^2m}\delta$.  If $e(G)\geq (1-\delta)\frac{mn}{r}$, then
one of the following two cases holds:
\begin{enumerate}
\item[(i)] $G$ has a double star (component) of order at least $\frac{m+n}{r}$.
\item[(ii)]
\begin{enumerate}
\item For all but at most $\alpha^{1/3}m$ exceptional vertices  $x\in X$ we have
\beq\label{ave1}\deg(x) > \avdeg(X,Y) - \alpha^{1/3} n,\eeq
\item and for all but at most $\beta^{1/3}n$ exceptional vertices $y\in Y$ we have
\beq\label{ave2}\deg(y) > \avdeg(Y,X) - \beta^{1/3} m.\eeq
\end{enumerate}

\end{enumerate}
\end{lemma}

\begin{proof}
Suppose $e(G)\geq (1-\delta)\frac{mn}{r}$, but $(i)$ does not hold.  So by Lemma \ref{andraslemma2} we have 
\beq\label{elek}\frac{mn}{r}> e(G) \geq (1-\delta)\frac{mn}{r},\eeq
Suppose that $k_X$ vertices $x\in X$ satisfy $\deg(x) \leq \avdeg(X,Y) -\alpha^{1/3} n$ and suppose that $k_Y$ vertices $y\in Y$ satisfy $\deg(y) \leq \avdeg(Y,X) -\beta^{1/3} n$. 
Denote the vertices in $X$ by $x_1, x_2, \ldots, x_m$, where $x_1, \ldots, x_{k_X}$ are the
exceptional vertices.  Likewise, denote the vertices in $Y$ by $y_1, y_2, \ldots, y_n$, where $y_1, \ldots, y_{k_Y}$ are the
exceptional vertices.

We will use the
``defect form'' of the Cauchy-Schwarz inequality (as in \cite{Sz} or in \cite{KSSz3}): if
$$\sum_{i=1}^{k_X} \deg(x_i) = \frac{k_X}{m}\sum_{i=1}^m \deg(x_i)
+\Delta_X = k_X \avdeg(X,Y) + \Delta_X  \qquad(k_X\leq m),$$
then
$$\sum_{i=1}^m \deg(x_i)^2\geq\frac1m
\left(\sum_{i=1}^m \deg(x_i)\right)^2+\frac{\Delta_X^2m}{k_X(m-k_X)}\geq \frac1m \left(\sum_{i=1}^m \deg(x_i)\right)^2+\frac{\Delta_X^2}{k_X}.$$
Analogously, we have 
$$\sum_{i=1}^{n} \deg(y_i)^2\geq \frac1n \left(\sum_{i=1}^n \deg(y_i)\right)^2+\frac{\Delta_Y^2}{k_Y}.$$
Note that 
\beq \label{DXDY}
|\Delta_X| \geq k_X \alpha^{1/3} n  ~\text{ and }~ |\Delta_Y| \geq k_Y \beta^{1/3} m.
\eeq

Then the average order of a double star in $G$ can be estimated as follows
\begin{align*}
\frac{1}{e(G)} \sum_{xy\in E(G)} (\deg(x) + \deg(y)) 
&= \frac{1}{e(G)} \left( \sum_{i=1}^m \deg(x_i)^2 + \sum_{j=1}^n \deg(y_j)^2 \right)\\
&\geq \frac{1}{e(G)} \left( \frac{e(G)^2}{m} + \frac{\Delta_X^2}{k_X} + \frac{e(G)^2}{n} + \frac{\Delta_Y^2}{k_Y}\right)\\
&= e(G) \frac{m+n}{mn} + \frac{\Delta_X^2}{k_X e(G)}+\frac{\Delta_Y^2}{k_Y e(G)}\\
&\geq \frac{1-\delta}{r}(m+n) + \frac{r k_X \alpha^{2/3}n}{m}+\frac{r k_Y \beta^{2/3} m}{n},
\end{align*}
where the last inequality holds by \eqref{elek} and \eqref{DXDY}.
Now if $k_X\geq \alpha^{1/3}m$ or $k_Y\geq \beta^{1/3}n$, then we have 
$$\frac{1-\delta}{r}(m+n) + \frac{r k_X \alpha^{2/3}n}{m}+\frac{r k_Y \beta^{2/3} m}{n}\geq \frac{m+n}{r},$$
a contradiction with the fact that $(i)$ does not hold.  So 
$k_X< \alpha^{1/3}m$ and $k_Y< \beta^{1/3}n$ and thus (ii) holds.
\end{proof}

From Lemma \ref{stabilitylemma} we can prove our main lemma.  Note that this lemma is very similar to the main lemma (Lemma 2.2) in \cite{GYS2}, but the proof is vastly simplified here and gives a slightly better degree estimate.

\begin{lemma}\label{main}
Assume that $\delta \leq \min\{\frac{n}{64 r^4(m+n)}, \frac{m}{64r(m+n)}\}$.  Under the conditions of Lemma \ref{stabilitylemma}, if $G$ does not contain a component of order at least $\frac{m+n}{r}$, then the following holds:
\begin{enumerate}
\item[($ii'$)] There are $r$ components $C_1, C_2, \dots, C_r$ such that for each $1\leq i \leq r$ we have the following properties:
\begin{enumerate}
\item $|C_i| < \frac{m+n}{r}$,
\item $|C_i\cap X| \geq \avdeg(Y,X) - \beta^{1/3} m$,
\item $|C_i\cap Y| \geq \avdeg(X,Y) - \alpha^{1/3} n$,
\item $|X\setminus \cup_{i=1}^r C_i| \leq \alpha^{1/3} m$,
\item $|Y\setminus \cup_{i=1}^r C_i| \leq \beta^{1/3} n$.
\end{enumerate}

\end{enumerate}
\end{lemma}

\begin{proof} Suppose that $G$ does not contain a component of order at least $\frac{m+n}{r}$.  So $(ii)$ holds in Lemma \ref{stabilitylemma}. We have to show that in this case $(ii')$ holds as well.  
Let $X'$ and $Y'$ be the non-exceptional vertices of $X$ and $Y$ respectively.  
We first show that there are at most $r$ components covering $X'$ all of which contain at least one vertex from $Y'$, and every vertex in $Y'$ has a neighbor in $X'$.  This implies that there are at most $r$ components covering $X'\cup Y'$ and thus all of $(a)-(e)$ holds.

Note that $\delta \leq \min\{\frac{n}{64 r^4(m+n)}, \frac{m}{64r(m+n)}\}$ implies 
\beq\label{alphabeta}
\alpha=\frac{m+n}{r^2n}\delta\leq \frac{1}{64r^6} ~~\text{ and }~~ \beta=\frac{m+n}{r^2m}\delta\leq \frac{1}{64r^3}.
\eeq

Suppose there are $r+1$ vertices from $X'$, every pair of which are in a different component.  Since $\delta(X',Y)\geq (\frac{1-\delta}{r}-\alpha^{1/3})n$, we have
\beq \label{compcover}
n\geq (r+1)(\frac{1-\delta}{r} - \alpha^{1/3})n\geq (r+1)(\frac{1}{r+1}+\frac{1}{2r(r+1)} - \alpha^{1/3})n.
\eeq
Where the second inequality holds since $\delta\leq \frac{1}{2(r+1)}$.  Furthermore, by \eqref{alphabeta} we have $\alpha^{1/3}\leq \frac{1}{4r^2}< \frac{1}{2r(r+1)}$ and thus $(r+1)(\frac{1}{r+1}+\frac{1}{2r(r+1) }- \alpha^{1/3})n>n$ contradicting \eqref{compcover}.

Each such component will contain a vertex of $Y'$ provided
$$\avdeg(X,Y) - \alpha^{1/3} n \geq \frac{1-\delta}{r} n - \alpha^{1/3} n > \beta^{1/3}n,$$
which in turn is true if
$$2 \beta^{1/3} \leq \frac{1}{2r},$$
which is true by \eqref{alphabeta}.

Now each vertex in $Y'$ will have a neighbor in $X'$ provided
$$\avdeg(Y,X) - \beta^{1/3} m \geq \frac{1-\delta}{r} m - \beta^{1/3}m > \alpha^{1/3}m,$$
which in turn is true if
$$2 \beta^{1/3} \leq \frac{1}{2r},$$
which is true by \eqref{alphabeta}.

Finally, note that if there were fewer than $r$ components, then $(d)-(e)$ would imply that there is a component of size at least 
$$\frac{(1-\alpha^{1/3})m+(1-\beta^{1/3})n}{r-1}\geq \frac{(1-\frac{1}{r})(m+n)}{r-1}=\frac{m+n}{r},$$ contradicting $(a)$.  So there are exactly $r$ components.
\end{proof}

\subsection{Large monochromatic component}

\begin{proof}[Proof of Theorem \ref{tetel}]

Let $$\gamma=\frac{(m/n)^3}{128r^5}$$ and let $G$ be an $(X,Y)$-bipartite graph with $|X|=m$, $|Y|=n$, $m\leq n$, and $\delta(X,Y) > \left( 1 - \gamma\right) n$ and $\delta(Y,X) > \left( 1 -\gamma\right) m$.  Consider an $r$-edge coloring of $G$ and suppose, for contradiction, that there is no monochromatic component of order at least $\frac{m+n}{r}$.  

Using the minimum degree condition, the number of edges in $G$ is at least $(1-\gamma)mn$.
Denote the monochromatic bipartite graphs induced by each of the $r$ colors by $G_1, \ldots, G_r$.
Then for each $1\leq i \leq r$ we have $e(G_i)<\frac{mn}{r}$, since
otherwise, by applying Lemma \ref{andraslemma2} to $G_i$, we have a monochromatic component of order at least $\frac{m+n}{r}$ in color $i$.  For all $i\in [r]$, define $\gamma_i$ by $e(G_i)=(1-\gamma_i)\frac{mn}{r}$.  So for all $1\leq i \leq r$ we have
\beq\label{G_i}
e(G_i)=(1-\gamma_i)\frac{mn}{r}\geq (1-r\gamma ) \frac{mn}{r}.\eeq
Indeed, otherwise the number of edges in $G$ would be less than
$$(1-r\gamma ) \frac{mn}{r} + (r-1) \frac{mn}{r} = (1-\gamma) m n,$$
a contradiction.

Using (\ref{G_i}), we can apply Lemma \ref{main} for each $G_i, 1\leq i \leq r$ with $\delta_i = \gamma_i$, $\alpha_i=\frac{m+n}{r^2n}\delta_i$, and $\beta_i=\frac{m+n}{r^2m}\delta_i$.  Since $\delta_i\leq r\gamma$ by \eqref{G_i} and since $m/n\leq 1$, we have 
$$
\delta_i\leq r\gamma=\frac{(m/n)^3}{128r^4}\leq \min\{\frac{n}{64 r^4(m+n)}, \frac{m}{64r(m+n)}\},
$$
and thus the conditions of Lemma \ref{main} are satisfied.  Furthermore, we have 
\beq\label{alpha}
\alpha_i=\frac{m+n}{r^2n}\delta_i\leq \frac{m+n}{r^2n} r\gamma=\frac{m+n}{n}\frac{(m/n)^3}{128r^6}\leq \frac{(m/n)^3}{8r^6}, 
\eeq
and
\beq\label{beta}
\beta_i=\frac{m+n}{r^2m}\delta_i\leq \frac{m+n}{r^2m} r\gamma=\frac{m+n}{m}\frac{(m/n)^3}{128r^6}\leq \frac{(m/n)^2}{8r^6}\leq \frac{1}{8r^6}.
\eeq
Also note that since $\gamma_i^2\leq r^2\gamma^2=\frac{(m/n)^6}{128^2r^8}\leq \frac{m+n}{r^2m}$ we have
\beq\label{gba}
\gamma_i\leq \left(\frac{m+n}{r^2m}\gamma_i\right)^{1/3}=\beta_i^{1/3}\leq \alpha_i^{1/3}.
\eeq

Since we cannot have $(i)$ in Lemma \ref{main},
we must have the $r$ components, $C^i_1, \dots, C^i_r$, described in $(ii')$ of Lemma \ref{main} for each $G_i$, call these the {\em main components}. Consider the remaining set of vertices, call them $Z_i$, not covered by the union of these main components. By $(ii')(a)$, $Z_i$ is non-empty and by $(ii')(d)$ and $(ii')(e)$,  $|Z_i\cap X|\leq \alpha_i^{1/3}m$ and $|Z_i\cap Y|\leq \beta_i^{1/3}n$.
Furthermore, for all $i, j\in[r]$, note that $(ii')(a), (b)$ and $(c)$ imply that $|C^i_j\cap X|\geq \frac{1-\gamma_i}{r}m-\beta_i^{1/3}m$ and $|C^i_j\cap Y|\geq \frac{1-\gamma_i}{r}n-\alpha_i^{1/3}n$ and consequently,
\beq\label{X}|C^i_j\cap X| =|C_j^i|-|C_j^i\cap Y|\leq \frac{m+n}{r}-\left(\frac{1-\gamma_i}{r}-\alpha_i^{1/3}\right)n= \frac{m}{r}+(\frac{\gamma_i}{r}+\alpha_i^{1/3})n, \eeq
and
\beq\label{Y} |C^i_j\cap Y| =|C_j^i|-|C_j^i\cap X| \leq \frac{m+n}{r}-\left(\frac{1-\gamma_i}{r}-\beta_i^{1/3}\right)m= \frac{n}{r}+(\frac{\gamma_i}{r}+\beta_i^{1/3})m.\eeq

Without loss of generality suppose $G_1$ is the majority color class; that is, $e(G_1)\geq e(G)/r\geq (1-\gamma)\frac{mn}{r}$ and say $G_1$ is red.  So $\gamma_1\leq \gamma$.  We have that $Z_1\neq \emptyset$, so first suppose $x\in X\cap Z_1$.  So the red degree of $x$ is bounded by $|Z_1\cap Y|\leq \beta_1^{1/3}n$.  So $x$ has, say blue degree at least $\frac{1-\gamma-\beta_1^{1/3}}{r-1}n$.  Suppose $G_2$ is blue and note that by \eqref{G_i}, we have $\gamma_1\leq \gamma_2\leq r\gamma$.  We have
\begin{align*}
\frac{1-\gamma_1-\beta_1^{1/3}}{r-1}n \geq \frac{1-2\beta_1^{1/3}}{r-1}n \geq \frac{1-1/r^2}{r-1}n=\frac{n}{r}+\frac{n}{r^2}> \frac{n}{r}+(\frac{\gamma_2}{r}+\beta_2^{1/3})m,
\end{align*}
where the last inequality holds since \eqref{beta} and \eqref{gba} give
$$\frac{\gamma_2}{r}+\beta_2^{1/3}< 2\beta_2^{1/3}\leq \frac{1}{r^2}.$$
So this implies that $x$ is a non-exceptional blue vertex and contained in a main component of $G_2$, but then \eqref{Y} is violated.  

Finally, suppose the main red components cover all of $X$ (i.e. $Z_1\cap X=\emptyset$) and let $y\in Z_1\cap Y$.  So $y$ has no red neighbors, and thus has say blue degree at least $\frac{1-\gamma}{r-1}m$.  As before, clearly $y$ is a non-exceptional blue vertex and thus contained in a main blue component.  But now \eqref{X} gives the following contradiction,  
\begin{align*}
\frac{m}{r}+(\frac{\gamma_2}{r}+\alpha_2^{1/3})n\geq \frac{1-\gamma}{r-1}m \geq  \frac{1-1/r^2}{r-1}m&=\frac{m}{r}+\frac{m}{r^2}> \frac{m}{r}+(\frac{\gamma_2}{r}+\alpha_2^{1/3})n,
\end{align*}
where the last inequality holds since \eqref{alpha} and \eqref{gba} give
\[\frac{\gamma_2}{r}+\alpha_2^{1/3}< 2\alpha_2^{1/3}\leq \frac{m/n}{r^2}.\qedhere\]
\end{proof}

\section{Additive minimum degree}\label{sec:add}

\begin{proof}[Proof of Theorem \ref{-n/8}]
Let $G$ be an $X,Y$-bipartite graph on $n$ vertices with $|Y|\geq |X|> n/4$, $\delta(X,Y)\geq |Y|-n/8$, and $\delta(Y,X)\geq |X|-n/8$.  Consider a 2-edge-coloring of $G$.  

First note that since $|Y|\geq n/2$, we have 
\[e(G)\geq |X|(|Y|-n/8)\geq |X|(|Y|-|Y|/4)=\frac{3}{4}|X||Y|.\]
So by Lemma \ref{andraslemma2}, a largest monochromatic, say blue, component $H_1$ satisfies $|H_1|\geq 3n/8$.  Let $X_1=H_1\cap X$, $Y_1=H_1\cap Y$, $X_2=X\setminus H_1$, $Y_2=Y\setminus H_1$.  Let $x_i=|X_i|$ and $y_i=|Y_i|$.   

We begin with the following claim.

\begin{claim}
If $x_1\geq |X|/2$, then $y_1\geq |Y|/2$, and if $y_1\geq |Y|/2$, then $x_1\geq |X|/2$.
\end{claim}

\begin{proof}
First suppose $x_1\geq |X|/2>n/8$.  Either $y_1\geq |Y|/2$ and we are done, or else $y_1<|Y|/2$ and every pair of vertices in $X_1$ has a common red neighbor in $Y_2$ and every vertex in $Y_2$ has a red neighbor in $X_1$ and thus we have a red component which is larger than $H_1$, a contradiction. 

Now suppose $y_1\geq |Y|/2$.  Either $x_1\geq |X|/2$ and we are done, or else $x_1<|X|/2$.  Now if $|Y|/2>n/4$, then every pair of vertices in $X_2$ has a common red neighbor in $Y_1$ and since $x_2>|X|/2>n/8$, every vertex in $Y_1$ has a red neighbor in $X_2$ and thus we have a red component which is larger than $H_1$.  So suppose $|Y|/2=n/4$, i.e. $|Y|=|X|=n/2$.  In this case, $x_2>|X|/2\geq n/4$ so every pair of vertices in $Y_1$ has a common red neighbor in $X_2$ and every vertex in $X_2$ has a red neighbor in $Y_1$, and thus we have a red component which is larger than $H_1$, a contradiction.
\end{proof}

By the Claim, we may assume that we are done unless $x_1<|X|/2$ and $y_1<|Y|/2$. In particular, this implies that every pair of vertices in $X_1$ has a common red neighbor in $Y_2$, so there is a red component $H_2$ covering $X_1$ and at least $y_2-n/8$ vertices of $Y_2$.  

If $x_1>n/8$, then every vertex in $Y_2$ has a red neighbor in $X_1$ and there is a red component larger than $H_1$, a contradiction.  

So suppose $x_1\leq n/8$.  This implies $x_2=|X|-x_1>n/8$ and $n/8+y_1\geq x_1+y_1=|H_1|\geq 3n/8$, which implies $y_1\geq n/4$.  Since $x_2>n/8$, every vertex in $Y_1$ has a red neighbor in $X_2$.  If $y_1>n/4$, then every pair of vertices in $X_2$ has a common red neighbor in $Y_1$ and thus we have a red component which is larger than $H_1$, a contradiction.  So suppose $y_1=n/4$ and further suppose that every vertex in $X_2$ has exactly $n/8$ neighbors in $Y_1$ (otherwise we would be done as in the previous sentence).  This implies from the degree condition that $[X_2, Y_2]$ is a complete bipartite graph.  If $x_2>n/4$, then every pair of vertices in $Y_1$ has a common red neighbor in $X_2$ and every vertex in $X_2$ has a red neighbor in $Y_1$, so we have a red component which is larger than $H_1$, a contradiction.  So suppose $x_2\leq n/4$.  This implies that $|Y|=n-x_1-x_2\geq 5n/8$ and since $y_1=n/4$, we have $y_2\geq 3n/8$. Recall that the red component $H_2$ covers $X_1$ and at least $y_2-n/8\geq n/4$ vertices of $Y_2$.  So if any vertex in $Y_2\cap H_2$ had a red neighbor in $X_2$, then since every vertex in $X_2$ has exactly $n/8$ red neighbors in $Y_1$, we have a red component on at least $x_1+1+y_2-n/8+n/8>x_1+y_1=|H_1|$ vertices, a contradiction.  So every vertex in $Y_2\cap H_2$ only has blue neighbors in $X_2$ and since $[X_2, Y_2]$ is a complete bipartite graph, this implies that we have a blue component on 
\[x_2+y_2-n/8\geq x_2+n/4>x_1+n/4=x_1+y_1=|H_1|\] vertices, a contradiction.  
\end{proof}

We close with the following problem and note that any value of $\alpha$ greater than $1/8$ would improve the bound given in Corollary \ref{7/8}.

\begin{problem}\label{-n/6}
Determine the largest value of $\alpha$ so that the following is true.  Let $G$ be an $X,Y$-bipartite graph on $n$ vertices with $|Y|\geq |X|> 2\alpha n$.  If $\delta(X,Y)\geq |Y|-\alpha n$ and $\delta(Y,X)\geq |X|-\alpha n$, then in every 2-coloring of the edges of $G$, there exists a monochromatic component on at least $n/2$ vertices.
\end{problem}

\noindent{\bf Acknowledgment. } The authors are grateful to Andr\'as Gy\'arf\'as for helpful conversations.  We also thank Thanaporn Sumalroj for pointing out a typo in the first version of this paper and we thank two anonymous referees for their helpful comments.

\end{document}